\documentclass[12pt]{article}
\usepackage{amsfonts}
\usepackage{amsmath}
\usepackage{fullpage}
\usepackage{graphicx}
\usepackage{url}
\usepackage{ifpdf}
\ifpdf\usepackage[breaklinks=true]{hyperref}\fi

\usepackage{amsthm}
\newtheorem{lemma}{Lemma}
\newtheorem{theorem}{Theorem}

\usepackage{caption}

\newcommand{\maps}{\colon} 
\newcommand{\define}[1]{\textbf{\boldmath{#1}}}
\newcommand{\RR}{\ensuremath{\mathbb{R}}}
\newcommand{\CC}{\ensuremath{\mathbb{C}}}
\newcommand{\HH}{\ensuremath{\mathbb{H}}}
\newcommand{\OO}{\ensuremath{\mathbb{O}}}

\newcommand{\KK}{\ensuremath{\mathbb{K}}}
\newcommand{\A}{\mathrm{A}}

\newcommand{\C}{\mathrm{C}}
\newcommand{\D}{\mathrm{D}}
\newcommand{\E}{\mathrm{E}}
\newcommand{\F}{\mathrm{F}}
\newcommand{\J}
	{\left( \begin{matrix} 0 & 1 \\ -1 & 0 \end{matrix} \right)} 
\newcommand{\M}[2]{#1^{#2\times#2}} 
\newcommand{\mat}[1]{\mathbf{#1}}
\newcommand{\AAA}{\mat{A}}
\newcommand{\II}{\mat{I}}
\newcommand{\MM}{\mat{M}}
\newcommand{\PP}{\mat{P}}
\newcommand{\QQ}{\mat{Q}}
\newcommand{\Rmat}{\mat{R}}
\newcommand{\XX}{\mat{X}}

\newcommand{\GG}{\mat{\Gamma}}
\renewcommand{\SS}{\mat{\hbox{\boldmath$\sigma$}}}
\renewcommand{\bar}[1]{\overline{#1}}
\renewcommand{\star}[1]{#1^*}
\newcommand{\inda}{p} 
\newcommand{\indb}{q} 
\newcommand{\indc}{r} 
\newcommand{\Cl}{\mathrm{C}\ell}
\newcommand{\SO}{\mathrm{SO}}
\newcommand{\SU}{\mathrm{SU}}
\newcommand{\SL}{\mathrm{SL}}
\newcommand{\Sp}{\mathrm{Sp}}
\newcommand{\Spin}{\mathrm{Spin}}
\newcommand{\Pin}{\mathrm{Pin}}
\renewcommand{\O}{\mathrm{O}}
\newcommand{\tr}{\mathrm{tr \, }}
\newcommand{\der}{\mathfrak{der}}
\newcommand{\End}{\mathrm{End}}

\newcommand{\sa}{\mathfrak{sa}}
\newcommand{\so}{\mathfrak{so}}
\newcommand{\su}{\mathfrak{su}}

\renewcommand{\sp}{\mathfrak{sp}}
\renewcommand{\sl}{\mathfrak{sl}}
\renewcommand{\aa}{\mathfrak{a}}

\newcommand{\cc}{\mathfrak{c}}
\newcommand{\dd}{\mathfrak{d}}
\newcommand{\f}{\mathfrak{f}}
\newcommand{\e}{\mathfrak{e}}

\newcommand{\vv}{\mathfrak{v}} 
\newcommand{\iso}{\cong}

\newcommand{\tensor}{\otimes}
\renewcommand{\Im}{\mathrm{Im}\,}
\renewcommand{\Re}{\mathrm{Re}\,}
\newcommand{\kap}{\kappa+\kappa'_+,\kappa'_-}
\newcommand{\halfang}{\hbox{\small$\frac\theta2$}}
\newcommand{\Vtwo}{\mathbf{V}_2}
\newcommand{\Vfour}{\mathbf{V}_4}
\newcommand{\Ksuper}{\mathcal{K}}
\newcommand{\Kfour}{\M{\Ksuper}{4}}
\newcommand{\KKprime}{\KK'\otimes\KK}

\begin{document}


\title{
	{\bfseries\boldmath The magic square of Lie groups: \\
		the $2\times2$ case}
}

\author{
	Tevian Dray \\[-2.5pt]
	\normalsize
	\textit{Department of Mathematics, Oregon State University,
		Corvallis, OR  97331 USA} \\[-2.5pt]
	\normalsize
	{\tt tevian{\rm @}math.oregonstate.edu} \\
	\and
	John Huerta \\[-2.5pt]
	\normalsize
	\textit{CAMGSD,
		Instituto Superior T\'ecnico,
		1049-001 Lisbon, PORTUGAL} \\
	\normalsize
	{\tt john.huerta@tecnico.ulisboa.pt} \\
	\and
	Joshua Kincaid \\[-2.5pt]
	\normalsize
	\textit{Department of Physics, Oregon State University,
		Corvallis, OR  97331 USA} \\[-2.5pt]
	\normalsize
	{\tt kincajos{\rm @}math.oregonstate.edu} \\
}

\date{\normalsize 3 May 2014}

\maketitle

\begin{abstract}
A unified treatment of the $2\times2$ analog of the Freudenthal-Tits magic
square of Lie groups is given, providing an explicit representation in terms
of matrix groups over composition algebras.
\end{abstract}

{\small
\textbf{keywords:}
division algebras; magic squares; orthogonal groups; Clifford algebras


\textbf{MSC:} 
22E46, 
17A35, 
15A66  
}


\section{Introduction}

The Freudenthal--Tits magic square~\cite{Freudenthal, Tits} is a $4\times4$
array of semisimple Lie algebras, whose rows and columns are labeled by
composition algebras.  It is magical not only because of its symmetry, but also
because, in the row or column labeled by the octonions or the split octonions,
the square produces four of the five exceptional Lie algebras: $\f_4$, $\e_6$,
$\e_7$ and~$\e_8$.  Several constructions of the magic square are known
\cite{Freudenthal, Tits, Vinberg, SudberyBarton}, all of which take a pair of
composition algebras and produce a Lie algebra.  They provide concise and
elegant constructions of exceptional Lie algebras, and show how the
exceptional Lie algebras are related to the octonions.

This paper forms part of an effort which aims to give a similarly concise and
elegant construction for the exceptional Lie groups, by building a `magic
square of Lie groups'; that is, we want a construction that takes two
composition algebras and produces a Lie group, without the intermediate step
of constructing the Lie algebra.  In this paper, we construct the `$2\times2$
magic square of Lie groups'.  At the Lie algebra level, the `$2\times2$ magic
square' proposed by Barton and Sudbery~\cite{SudberyBarton} is a simpler
cousin of the Freudenthal--Tits magic square, so named because the $3\times3$
matrices used in constructing the usual magic square are replaced by
$2\times2$ matrices.  We emphasize that the labels `$2\times2$' and
`$3\times3$' used throughout this paper refer to the size of the underlying
matrices, and not to the magic squares themselves (which are $4\times4$).

Unlike the original `$3\times3$ magic square', the $2\times2$ magic square
contains no exceptional Lie algebras.  Instead, it consists of special
orthogonal algebras with various signatures.  It serves as a kind of test case
for a similar analysis of the $3\times3$ magic square, since it involves the
noncommutativity of the quaternions and nonassociativity of the octonions
without the further complexity of the exceptional Lie algebras.  Moreover, it
has an intriguing connection to string theory that makes it of interest in its
own right: the first three rows give, in succession, the infinitesimal
rotational, Lorentz, and conformal symmetries of the Minkowski spacetimes
where the classical superstring can be defined.  The octonionic column
corresponds to 10-dimensional spacetime, where the superstring can also be
quantized.

\begin{table}
\begin{center}
\begin{tabular}{|c|c|c|c|c|}
\hline
& $\RR$ & $\CC$ & $\HH$ & $\OO$ \\ 
\hline
$\RR'$ & $\so(3)$ & $\su(3)$ & $\sp(3)$ & $\f_4$ \\
\hline
$\CC'$ & $\sl(3,\RR)$ & $\sl(3,\CC)$ & $\aa_{5(-7)}$ & $\e_{6(-26)}$ \\
\hline
$\HH'$ & $\sp(6,\RR)$ & $\su(3,3) $ & $\dd_{6(-6)}$ & $\e_{7(-25)}$ \\
\hline
$\OO'$ & $\f_{4(4)}$ & $\e_{6(2)}$ & $\e_{7(-5)}$ & $\e_{8(-24)}$ \\
\hline
\end{tabular}
\end{center}
\caption{The $3\times3$ half-split magic square.}
\label{3x3alg}
\end{table}

Our interest in this paper is in the `half-split' magic square, with columns
labeled by normed division algebras and rows by split composition algebras.
To see the patterns we want to explore, first consider the \define{half-split
$3\times3$ magic square} shown in Table~\ref{3x3alg}.
Here, $\sp(3)$ denotes the compact real form of $\cc_3$, whereas $\sp(6,\RR)$
denotes the Lie algebra respecting the usual symplectic form on $\RR^6$.  A
number in parentheses is the signature of the Killing form, which is the
excess of plus signs (``boosts'') over minus signs (``rotations'') in the
diagonalization of this form.  As is well known, the Dynkin diagram and
signature specify a real form completely.


Perhaps the most concise construction of the magic square is due to Vinberg.
Given a pair of composition algebras $\KK'$ and $\KK$, Vinberg's construction
\cite{Vinberg} says the corresponding entry of the magic square will be 
\begin{equation}
\vv_3(\KK',\KK) = \sa_3(\KKprime) \oplus \der(\KK') \oplus \der(\KK) .
\label{Vin3}
\end{equation} 
Here, $\sa_3(\KKprime)$ denotes the set of traceless anti-Hermitian $3\times3$
matrices, $\der(\KK')$ and $\der(\KK)$ are the Lie algebras of derivations on
the composition algebras $\KK'$ and $\KK$, and their sum is a Lie subalgebra.
Since our focus is on the $2\times2$ magic square in this paper, we will not
need to describe the bracket on $\vv_3(\KK',\KK)$, which is given by a
complicated formula that can be found in Barton and
Sudbery~\cite{SudberyBarton}.

Now make note of the pattern in the first two columns of the magic square.  In
what follows, $\KK$ denotes $\RR$ or $\CC$, $\M{\KK}{n}$ denotes the set of $n
\times n$ matrices with entries in $\KK$, and $X^\dagger = \overline{X}^T$, the
conjugate transpose of the matrix $X$.  We observe that:
\begin{itemize} 
\item
In the first row, $\so(3)$ and $\su(3)$ are both Lie algebras of traceless,
anti-Hermitian matrices.  If we define
\begin{equation}
\su(3,\KK) = \{ X\in\M{\KK}{3}\,: \,X^\dagger=-X, \,\tr X=0 \} .
\end{equation}
for $\KK=\RR,\CC$, then $\su(3,\RR)$ is $\so(3)$ and $\su(3,\CC)$ is $\su(3)$.
\item
In the second row, $\sl(3,\RR)$ and $\sl(3,\CC)$ are both Lie algebras of
traceless matrices, that is, they are special cases of
\begin{equation}
\sl(3,\KK) = \{ X\in\M{\KK}{3}\,: \,\tr X=0 \} .
\end{equation}
for $\KK=\RR,\CC$.
\end{itemize}
We can carry our observations further if we note that $\su(3,3)$ preserves an
inner product on~$\CC^6$ that, in a suitable basis, bears a striking
resemblance to a symplectic form:
\begin{equation}\omega(x,y) = x^\dagger \J y ,
\end{equation}
where we regard $x,y\in\CC^6$ as column vectors.  The only difference between
$\omega$ and the usual symplectic structure is that $\omega$ is conjugate
linear in its first slot.  Thus, we see that:
\begin{itemize}
\item
In the third row, $\sp(6,\RR)$ and $\su(3,3)$ are both Lie algebras of the
form
\footnote{We emphasize that $\sp(6,\CC)$ is \emph{not} the usual symplectic
Lie algebra, due to the use of Hermitian conjugation rather than transpose in
its definition.}
\begin{equation}
\sp(6,\KK) = \{ X\in\M{\KK}{6}\,: \,X^\dagger J+JX=0, \,\tr X=0 \}
\end{equation}
for $\KK=\RR,\CC$, where $J$ is the $6\times6$ matrix with block
decomposition $J=\J$.
\end{itemize}

Barton and Sudbery showed how to extend these patterns across the first three
rows by giving definitions of Lie algebras $\su(n,\KK)$, $\sl(n,\KK)$ and
$\sp(2n,\KK)$ that work when $\KK$ is any normed division algebra, provided
$n\leq3$, and for any $n$ when $\KK$ is associative.
\footnote{Barton and Sudbery write $\sa(n,\KK)$ for the Lie algebra we write as
$\su(n,\KK)$.  Moreover, their $\sp(2n,\CC)$ is again \emph{not} the
symplectic algebra, but instead denotes the Lie algebra usually called
$\su(n,n)$.}

\begin{table}
\begin{center}
\begin{tabular}{|c|c|c|c|c|}
\hline
& $\RR$ & $\CC$ & $\HH$ & $\OO$ \\ 
\hline
$\RR'$ & $\su(3,\RR)$ & $\su(3,\CC)$ & $\su(3,\HH)$ & $\su(3,\OO)$ \\
\hline
$\CC'$ & $\sl(3,\RR)$ & $\sl(3,\CC)$ & $\sl(3,\HH)$ & $\sl(3,\OO)$ \\
\hline
$\HH'$ & $\sp(6,\RR)$ & $\sp(6,\CC)$ & $\sp(6,\HH)$ & $\sp(6,\OO)$ \\
\hline
\end{tabular}
\end{center}
\caption{The $3\times3$ magic square, first three rows according to Barton
and Sudbery.}
\label{3x3BS}
\end{table}

When $n=3$, the above algebras reproduce the first three rows of the
$3\times3$ magic square, as shown in Table~\ref{3x3BS}.
Of particular interest, the exceptional Lie algebras are:
\begin{equation}
\su(3,\OO) = \f_4, \quad
\sl(3,\OO) = \e_{6(-26)},
\quad \sp(6,\OO) = \e_{7(-25)} . 
\end{equation}

On the other hand, when $n = 2$, $\su(2,\KK)$, $\sl(2,\KK)$ and $\sp(4,\KK)$
turn out to be orthogonal Lie algebras, namely
\begin{equation}
\su(2,\KK) = \so(\RR'\oplus\KK), \quad
\sl(2,\KK) = \so(\CC'\oplus\KK), \quad
\sp(4,\KK) = \so(\HH'\oplus\KK),
\end{equation}
where the direct sums above are orthogonal direct sums. This leads Barton and
Sudbery to take the \define{half-split $2\times2$ magic square} to be the
square with entry $\so(\KK'\oplus\KK)$ for any split composition algebra
$\KK'$ and normed division algebra $\KK$, as shown in Table~\ref{2x2alg}.
The given signatures follow from adding the signatures of $\KK'$ and $\KK$ in
the orthogonal direct sum.  We will delve further into the properties of
composition algebras later.

\begin{table}
\begin{center}
\begin{tabular}{|c|c|c|c|c|}
\hline
& $\RR$ & $\CC$ & $\HH$ & $\OO$ \\ 
\hline
$\RR'$ & $\so(2)$ & $\so(3)$ & $\so(5)$ & $\so(9)$ \\
\hline
$\CC'$ & $\so(2,1)$ & $\so(3,1)$ & $\so(5,1)$ & $\so(9,1)$ \\
\hline
$\HH'$ & $\so(3,2)$ & $\so(4,2)$ & $\so(6,2)$ & $\so(10,2)$ \\
\hline
$\OO'$ & $\so(5,4)$ & $\so(6,4)$ & $\so(8,4)$ & $\so(12,4)$ \\
\hline
\end{tabular}
\end{center}
\caption{The $2\times2$ magic square.}
\label{2x2alg}
\end{table}


Despite its different appearance, this $2\times2$ magic square really
\emph{is} a cousin of the $3\times3$ magic square.  Barton and Sudbery prove
that each entry of this magic square is given by a construction similar to
Vinberg's, namely
\begin{equation}
\vv_2(\KK',\KK) = \sa_2(\KKprime) \oplus \so(\Im\KK') \oplus \so(\Im\KK) .
\end{equation}
Now, $\sa_2(\KKprime)$ denotes the set of traceless, anti-Hermitian $2\times2$
matrices over $\KKprime$, while $\Im\KK'$ and $\Im\KK$ denote the `imaginary
parts' of $\KK'$ and $\KK$, respectively.  In contrast to Vinberg's
construction of the $3\times3$ magic square, the algebras of derivations have
been replaced with the orthogonal algebras $\so(\Im\KK')$ and $\so(\Im\KK)$.
However, just as for the $3\times3$ magic square, the first three rows can be
expressed in terms of (generalized) unitary, linear, and symplectic algebras,
as shown in Table~\ref{2x2BS}; compare Table~\ref{3x3BS}.

\begin{table}[b]
\begin{center}
\begin{tabular}{|c|c|c|c|c|}
\hline
& $\RR$ & $\CC$ & $\HH$ & $\OO$ \\ 
\hline
$\RR'$ & $\su(2,\RR)$ & $\su(2,\CC)$ & $\su(2,\HH)$ & $\su(2,\OO)$ \\
\hline
$\CC'$ & $\sl(2,\RR)$ & $\sl(2,\CC)$ & $\sl(2,\HH)$ & $\sl(2,\OO)$ \\
\hline
$\HH'$ & $\sp(4,\RR)$ & $\sp(4,\CC)$ & $\sp(4,\HH)$ & $\sp(4,\OO)$ \\
\hline
\end{tabular}
\end{center}
\caption{The $2\times2$ magic square, first three rows.}
\label{2x2BS}
\end{table}

Of particular interest, the octonionic column becomes:
\begin{equation}
\su(2,\OO) = \so(9), \quad \sl(2,\OO) = \so(9,1), \quad \sp(4,\OO) = \so(10,2) .
\end{equation}
These are, respectively, the Lie algebras of infinitesimal rotations, Lorentz
transformations, and conformal transformations for Minkowski spacetime
$\RR^{9,1}$, which is of precisely the dimension where string theory can be
quantized.  This intriguing connection to the octonions is not a coincidence
\cite{FairlieI, Schray, BHsuperI}, but is far from fully understood.

Dray, Manogue and their collaborators have worked steadily to lift Barton and
Sudbery's construction of the Lie algebras $\su(n,\OO)$, $\sl(n,\OO)$ and
$\sp(2n,\OO)$ to the group level.  In the case $n=2$, Manogue and Schray
\cite{Lorentz} gave an explicit octonionic representation of the Lorentz group
$\SO(9,1)$ in 10 spacetime dimensions, and later Manogue and Dray
\cite{Dim,Spin} outlined the implications of this mathematical description for
the description of fundamental particles.  In brief, Manogue and Schray
constructed a group that deserves to be called $\SL(2,\OO)$ that was the
double cover of (the identity component of) $\SO(9,1)$, that is:
\begin{equation}
\SL(2,\OO) \equiv \SO(9,1) .
\label{so91}
\end{equation}
Here we use the symbol ``$\equiv$'' to mean ``isomorphic up to cover''---that
is, we will write $G \equiv H$ to mean the Lie groups $G$ and $H$ have the
same Lie algebra.  Moving one step up in the magic square, if we define
$\SU(2,\OO)$ to be the maximal compact subgroup of $\SL(2,\OO)$, we also get:
\begin{equation}
\SU(2,\OO) \equiv \SO(9) .
\end{equation}
Because all other division algebras are subalgebras of the octonions, these
two constructions fully capture the first two rows of the \define{$2\times2$
magic square of Lie groups} shown in Table~\ref{2x2gp}.

\begin{table}
\begin{center}
\begin{tabular}{|c|c|c|c|c|}
\hline
& $\RR$ & $\CC$ & $\HH$ & $\OO$ \\
\hline
$\RR'$ & $\SU(2,\RR) \equiv \SO(2)$ & $\SU(2,\CC) \equiv \SO(3)$ &
$\SU(2,\HH) \equiv \SO(5)$ & $\SU(2,\OO) \equiv \SO(9)$ \\
\hline
$\CC'$ & $\SL(2,\RR) \equiv \SO(2,1)$ & $\SL(2,\CC) \equiv \SO(3,1)$ &
$\SL(2,\HH) \equiv \SO(5,1) $ & $\SL(2,\OO) \equiv \SO(9,1)$ \\
\hline
$\HH'$ & $\Sp(4,\RR) \equiv \SO(3,2)$ & $\SU(2,2) \equiv \SO(4,2)$ &
$\SO(6,2)$ & $\SO(10,2)$ \\
\hline
$\OO'$ & $\SO(5,4)$ & $\SO(6,4)$ & $\SO(8,4)$ & $\SO(12,4)$ \\
\hline
\end{tabular}
\end{center}
\caption{The $2\times2$ magic square of Lie groups.}
\label{2x2gp}
\end{table}

More recently, Dray and Manogue~\cite{Denver,York} have extended these results
to the exceptional Lie group $\E_6$, using the framework described in more
detail by Wangberg and Dray~\cite{Structure,Sub} and in Wangberg's
thesis~\cite{AaronThesis}.  All of these results rely on the description of
certain groups using matrices over division algebras.  Just as $\SL(2,\OO)$
appears in the second row and last column of the $2\times2$ magic square of
Lie groups, $\E_6$ appears in the corresponding spot of the $3\times3$ magic
square.  Using $\SL(2,\OO)$ to bootstrap the process, Dray, Manogue and
Wangberg define a group that deserves to be called $\SL(3,\OO)$ and prove that:
\begin{equation}
\SL(3,\OO) \equiv \E_{6(-26)}
\label{e6}
\end{equation}
where, again, we take the symbol $\equiv$ to mean ``isomorphic up to cover''.
As before, if we take $\SU(3,\OO)$ to be the maximal compact subgroup of
$\SL(3,\OO)$, we immediately obtain:
\begin{equation}
\SU(3,\OO) \equiv \F_4 .
\end{equation}
Once again, because all other normed division algebras are subalgebras of the
octonions, we obtain the first two rows of the \define{$3\times3$ magic
square of Lie groups}, as shown in Table~\ref{3x3gp}.

\begin{table}
\begin{center}
\begin{tabular}{|c|c|c|c|c|}
\hline
& $\RR$ & $\CC$ & $\HH$ & $\OO$ \\
\hline
$\RR'$ & $\SU(3,\RR)$ & $\SU(3,\CC)$ & $\SU(3,\HH) \equiv \C_3$ &
$\SU(3,\OO) \equiv \F_4$ \\
\hline
$\CC'$ & $\SL(3,\RR)$ & $\SL(3,\CC)$ & $\SL(3,\HH) \equiv \A_{5(-7)}$ &
$\SL(3,\OO) \equiv \E_{6(-26)}$ \\
\hline
$\HH'$ & $\Sp(6,\RR) \equiv \C_{3(3)}$ & $\SU(3,3)$ & $\D_{6(-6)}$ &
$\E_{7(-25)}$ \\
\hline
$\OO'$ & $\F_{4(4)}$ & $\E_{6(2)}$ & $\E_{7(-5)}$ & $\E_{8(-24)}$ \\
\hline
\end{tabular}
\end{center}
\caption{The $3\times3$ magic square of Lie groups.}
\label{3x3gp}
\end{table}

The ultimate goal of this project is to extend the above descriptions from the
first two rows of the magic squares to the remaining two rows, culminating in
new constructions of the exceptional Lie groups $\E_7$ and $\E_8$.  An
additional step in this direction was recently taken by Dray, Manogue, and
Wilson~\cite{Denver2}, who showed that
\begin{equation}
\Sp(6,\OO) \equiv \E_{7(-25)}
\end{equation}
and along the way also that
\begin{equation}
\Sp(4,\OO) \equiv \SO(10,2) ,
\end{equation}
thus completing the interpretation of the third row in both Lie group magic
squares; Wilson~\cite{Wilson} has also recently given a quaternionic
construction of $\E_7$.  But what about the fourth row?

In this paper, we take a different approach, and develop some tools for
working with the entire $2\times2$ magic square at once.  At the Lie algebra
level, recall that this magic square consists of the orthogonal algebras
$\so(\KK\oplus\KK')$, where ``$\oplus$'' denotes the orthogonal direct sum.
We will show how to use composition algebras to talk about the corresponding
Lie groups, in two different ways.

First, using composition algebras, we will construct a module of the Clifford
algebra $\Cl(\KK\oplus\KK')$ on the space of $4\times4$ matrices with
entries in $\KKprime$.  In the standard way, this gives a representation of
$\Spin(\KK\oplus\KK')$ on $\M{(\KKprime)}{4}$.  Identifying a certain
subspace of the $4\times4$ matrices, $\M{(\KKprime)}{4}$, with $\KK \oplus
\KK'$, this representation will restrict to the usual representation of
$\Spin(\KK\oplus\KK')$ on $\KK\oplus\KK'$.

We will then show that each group in the $2\times2$ magic square can be
written in the form $\SU(2,\KKprime)$.  Kincaid and
Dray~\cite{JoshuaThesis,so42} took the first step in providing a composition
algebra description of the third row of the magic squares by showing that
$\SO(4,2) \equiv \SU(2,\HH'\otimes\CC)$.  We extend their work by showing that
$\Spin(\KK\oplus\KK')$ acts on $\KK\oplus\KK'$ just as $\SU(2,\CC)$ acts on
the space of $2\times2$ Hermitian matrices.  We therefore rechristen
$\Spin(\KK\oplus\KK')$ as $\SU(2,\KKprime)$ when working with this
representation.

\section{Composition Algebras}


\begin{figure}
\centering
\begin{minipage}{0.41\textwidth}
  \centering
  \includegraphics[width=5cm]{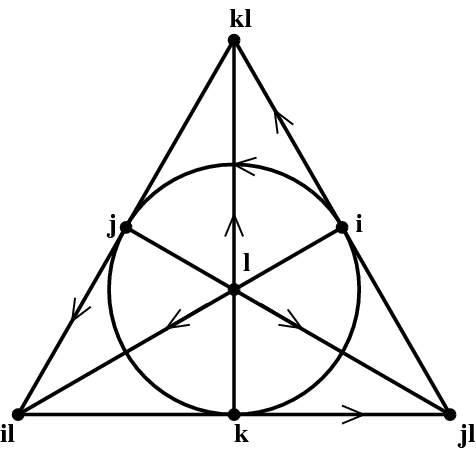}
  \captionof{figure}
    {A graphical representation of the  octonionic multiplication table.}
  \label{omult3}
\end{minipage}
\qquad
\begin{minipage}{0.5\textwidth}
  \centering
  \small
  \begin{tabular}[b]{|c|c|c|c|c|c|c|c|}
  \hline
  &\boldmath$i$&\boldmath$j$&\boldmath$k$&\boldmath$k\ell$
    &\boldmath$j\ell$&\boldmath$i\ell$&\boldmath$\ell$\\\hline
  \boldmath$i$&$-1$&$k$&$-j$&$j\ell$&$-k\ell$&$\ell$&$i\ell$\\\hline
  \boldmath$j$&$-k$&$-1$&$i$&$-i\ell$&$\ell$&$k\ell$&$j\ell$\\\hline
  \boldmath$k$&$j$&$-i$&$-1$&$-\ell$&$i\ell$&$-j\ell$&$k\ell$\\\hline
  \boldmath$k\ell$&$-j\ell$&$i\ell$&$\ell$&$-1$&$i$&$-j$&$-k$\\\hline
  \boldmath$j\ell$&$k\ell$&$\ell$&$-i\ell$&$-i$&$-1$&$k$&$-j$\\\hline
  \boldmath$i\ell$&$\ell$&$-k\ell$&$j\ell$&$j$&$-k$&$-1$&$-i$\\\hline
  \boldmath$\ell$&$-i\ell$&$-j\ell$&$-k\ell$&$k$&$j$&$i$&$-1$\\\hline
  \noalign{\vspace{0.15in}}
  \end{tabular}
  \captionof{table}{The octonionic multiplication table.}
  \label{omult}
\end{minipage}
\end{figure}

\begin{table}
\centering
\small
\begin{tabular}{|c|c|c|c|c|c|c|c|}
\hline
$$&\boldmath$I$&\boldmath$J$&\boldmath$K$
  &\boldmath$KL$&\boldmath$JL$&\boldmath$IL$&\boldmath$L$\\\hline
\boldmath$I$&$-1$&$K$&$-J$&$JL$&$-KL$&$-L$&$IL$\\\hline
\boldmath$J$&$-K$&$-1$&$I$&$-IL$&$-L$&$KL$&$JL$\\\hline
\boldmath$K$&$J$&$-I$&$-1$&$-L$&$IL$&$-JL$&$KL$\\\hline
\boldmath$KL$&$-JL$&$IL$&$L$&$1$&$-I$&$J$&$K$\\\hline
\boldmath$JL$&$KL$&$L$&$-IL$&$I$&$1$&$-K$&$J$\\\hline
\boldmath$IL$&$L$&$-KL$&$JL$&$-J$&$K$&$1$&$I$\\\hline
\boldmath$L$&$-IL$&$-JL$&$-KL$&$-K$&$-J$&$-I$&$1$\\\hline
\end{tabular}
\caption{The split octonionic multiplication table.}
\label{smult}
\end{table}

A \define{composition algebra} $\KK$ is a nonassociative real algebra with a
multiplicative unit 1 equipped with a nondegenerate quadratic form $Q$
satisfying the \define{composition property}:
\begin{equation}
Q(xy) = Q(x) Q(y), \quad x,y \in \KK.
\end{equation} 
A composition algebra for which $Q$ is positive definite is called a
\define{normed division algebra}.  On the other hand, when $Q$ is indefinite,
$\KK$ is called a \define{split composition algebra}.  In the latter case, it
was shown by Albert~\cite{Albert} that the quadratic form $Q$ must be
`split'.  Recall that the \define{signature} of a quadratic form is the excess
of plus signs over minus signs in its diagonalization.  A nondegenerate
quadratic form on a real vector space is \define{split} if its signature is as
close to 0 as possible: 0 for an even dimensional space, and $\pm 1$ for an
odd dimensional space.

By a theorem of Hurwitz~\cite{Hurwitz}, there are exactly four normed division
algebras: the real numbers $\RR$, the complex numbers $\CC$, the quaternions
$\HH$, and the octonions $\OO$.  Similarly, there are exactly four split
composition algebras: the real numbers
\footnote{The real numbers appear in both lists, as only a one-dimensional
space can have a quadratic form both positive definite and split.}
$\RR'=\RR$, the split complex numbers $\CC'$, the split quaternions $\HH'$,
and the split octonions $\OO'$.  In either case, these algebras have
dimensions 1, 2, 4, and 8, respectively.

Let us sketch the construction of the normed division algebras and their split
cousins.  Because the octonions and the split octonions contain all the other
composition algebras as subalgebras, we will invert the usual order and
construct them first. 

The \define{octonions} $\OO$ are the real algebra spanned by the multiplicative
unit 1 and seven square roots of $-1$:
\begin{equation}
\OO = \mathrm{span}\{ 1,i,j,k,k\ell,j\ell,i\ell,\ell \} .
\end{equation}
The basis elements besides $1$ are called \define{imaginary units}.  The
products of these imaginary units are best encapsulated in a figure known as
the Fano plane, equipped with oriented edges, as shown in Figure~\ref{omult3}.
Here, the product of any two elements is equal to the third element on the same
edge, with a minus sign if multiplying against orientation.  For instance:
\begin{equation}
j(i\ell) = k\ell = - (i\ell)j .
\end{equation}
As we alluded to above, the square of any imaginary unit is $-1$.  These rules
suffice to multiply any pair of octonions; the imaginary units
$k\ell,j\ell,i\ell$ are precisely the products suggested by their names.  The
full multiplication table is given in Table~\ref{omult}.

All other normed division algebras are subalgebras of $\OO$.  The real numbers
$\RR$ are the subalgebra spanned by $1$, the \define{complex numbers}
$\CC$ are the subalgebra spanned by $\{1, i\}$, and the \define{quaternions}
$\HH$ are the subalgebra spanned by $\{1,i,j,k\}$.  Of course, there are many
other copies of $\CC$ and $\HH$ in $\OO$.
This construction can be reversed, using the Cayley--Dickson
process~\cite{Schafer}; as vector spaces, we have
\begin{equation}
\CC=\RR\oplus\RR i, \quad \HH=\CC\oplus\CC j, \quad \OO=\HH\oplus\HH\ell .
\end{equation}

\define{Conjugation} is the linear map on $\OO$ which fixes 1 and sends every
imaginary unit to its negative.  It restricts to an operation on $\RR$, $\CC$
and $\HH$, also called conjugation, which is trivial on $\RR$, and coincides
with the usual conjugation on $\CC$ and $\HH$.  For an arbitrary octonion
$x\in\OO$, we write its conjugate as $\bar{x}$.  We define the \define{real}
and \define{imaginary} part of $x$ with the usual formulas,
\begin{equation}
\Re(x) = \frac{x + \bar{x}}{2}, \quad \Im(x) = \frac{x - \bar{x}}{2} ,
\end{equation}
and we say that $x$ is \define{real} or \define{imaginary} if it is equal to
its real or imaginary part, respectively.  The set of all imaginary octonions
is denoted $\Im\OO$.  Our notation and terminology for the other normed
division algebras is similar.

We can show that for a pair of octonions $x, y \in \OO$, conjugation satisfies
$\bar{xy} = \bar{y} \> \bar{x}$.  The quadratic form on $\OO$ is defined by:
\begin{equation}
Q(x) = x\bar{x} = \bar{x}x .
\end{equation}
We will also write $Q(x)$ as $|x|^2$.  Polarizing, we see the quadratic form
comes from the inner product:
\begin{equation}
(x,y) = \Re(x\bar{y}) = \Re(\bar{x}y).
\end{equation}
Moreover, a straightforward calculation shows that $1$ and the imaginary units
are orthonormal with respect to this inner product.  Explicitly, if
\begin{equation}
a = a_1 1 + a_2 i +  a_3 j + a_4 k + a_5 k\ell + a_6 j \ell + a_7 i \ell + a_8 \ell\end{equation}
we have
\begin{equation}
|a|^2 = a_1^2 + a_2^2 + a_3^2 +  a_4^2 + a_5^2 + a_6^2 + a_7^2 + a_8^2,
\end{equation}
so the quadratic form is positive definite.  Finally, it follows from the
definition that the quadratic form satisfies the composition property:
\begin{equation}
|xy|^2 = |x|^2 |y|^2, \quad  x,y \in \OO .
\end{equation}
Thus, $\OO$ is a normed division algebra, as promised.  The quadratic form and
inner product restrict to the other normed division algebras, and we use the
same notation.

The \define{split octonions} $\OO'$ are the real algebra spanned by the
multiplicative unit 1 and three square roots of $-1$, and four square roots of
$+1$:
\begin{equation}
\OO' = \mathrm{span}\{1,I,J,K,KL,JL,IL,L\} .
\end{equation}
The basis elements besides $1$ are again called \define{imaginary units}.  The
products of these imaginary units are given in Table~\ref{smult}.

All other split composition algebras are subalgebras of $\OO'$.  The
\define{split real numbers $\RR'$} are the subalgebra spanned by $1$, the
\define{split complex numbers} $\CC'$ are the subalgebra spanned by $\{1,
L\}$, and the \define{split quaternions} $\HH'$ are the subalgebra spanned by
$\{1,L,K,KL\}$.  Of course, there are many other copies of $\CC'$ and $\HH'$
in $\OO'$.  Finally, the split real numbers, split complex numbers and split
quaternions have more familiar forms, namely
\begin{equation}
\RR' = \RR, \quad \CC' \iso \RR\oplus\RR, \quad \HH' \iso \M{\RR}{2} .
\end{equation}
In other words, the split reals are just the reals, the split complexes are
isomorphic to the algebra $\RR\oplus\RR$ with multiplication and addition
defined componentwise, and the split quaternions are isomorphic to the algebra
of real $2\times2$ matrices.  Again, this construction can be reversed using
the Cayley--Dickson process; as vector spaces, we have
\begin{equation}
\CC'=\RR\oplus\RR L, \quad \HH'=\CC\oplus\CC L, \quad \OO'=\HH\oplus\HH L
\end{equation}
(where these copies of $\RR$, $\CC$, and $\HH$ live in $\OO'$, not $\OO$).

Conjugation, real part and imaginary part are defined in exactly the same way
for $\OO'$ as for $\OO$, but we will write the conjugate of $X \in \OO'$ as
$\star{X}$. The quadratic form on $\OO'$ is:
\begin{equation}
Q(X) = X\star{X} = \star{X}X .
\end{equation}
We will also write this form as $|X|^2$, even though it is not positive
definite.  Polarizing, we see the quadratic form comes from the inner product:
\begin{equation}
(X,Y) = \Re(X \star{Y}) = \Re(\star{X} Y).
\end{equation}
Moreover, a straightforward calculation shows that $1$ and the imaginary units
are orthogonal with respect to this inner product.  Explicitly, if
\begin{equation}
A = A_1 1 + A_2 I + A_3 J + A_4 K + A_5 KL + A_6 J L + A_7 I L + A_8 L
\end{equation}
we have
\begin{equation}
|A|^2 = A_1^2 + A_2^2 + A_3^2 +  A_4^2 - A_5^2 - A_6^2 - A_7^2 - A_8^2 ,
\end{equation}
so the quadratic form has split signature.  Finally, it follows from the
definition that the quadratic form satisfies the composition property.  Thus
$\OO'$ is a split composition algebra, as claimed.  The quadratic form and
inner product restrict to the other split composition algebras, and we use the
same notation.

As is well known, the octonions are not associative, but they are
\define{alternative}.  This means that any triple product of two elements
associates:
\begin{equation}
(xx)y = x(xy), \quad (xy)x = x(yx), \quad (yx)x = y(xx), \quad x,y \in \OO .
\end{equation}
Equivalently, by Artin's theorem~\cite{Schafer}, any subalgebra generated by
at most two elements is associative.  These relations also hold for the split
octonions, and trivially in the other composition algebras, which are
associative.

In what follows, we will work with the algebra $\OO'\otimes\OO$ and its
subalgebras $\Ksuper=\KKprime$, where $\KK$ is any of the division algebras
$\RR$, $\CC$, $\HH$, $\OO$, and $\KK'$ any of their split versions.
Multiplication in $\KKprime$ is defined in the usual way:
\begin{equation}
(A \tensor a) (B \tensor b) = AB \tensor ab ,
\end{equation}
for $A \tensor a, B \tensor b \in \KKprime$.  Conjugation in $\KKprime$ is
defined to conjugate each factor:
\begin{equation}
\bar{A \tensor a} = \star{A} \tensor \bar{a} . 
\end{equation}
We let $\kappa=|\KK|=1,2,4,8$, and for $\KK'$ we keep track separately of the
number of positive-normed basis units, $\kappa'_+=1,1,2,4$, and
negative-normed basis units, $\kappa'_-=0,1,2,4$, with
$\kappa'_++\kappa'_-=|\KK'|$.

\section{\boldmath The Clifford Algebra $\Cl(\kap)$}
\label{Clifford}

We now introduce our principal tool: a representation of the Clifford algebra
$\Cl(\kap)$ using matrices over composition algebras.  Because Clifford
algebras can be used to construct spin groups in a well-known fashion, this
will allow us to construct the groups of the $2\times2$ magic square.

To begin, let us write the vector space $\KK'\oplus\KK$ using $2\times2$
matrices:
\begin{equation}
\label{defx}
\Vtwo = \left\{ \begin{pmatrix}A&\bar{a}\\a&-\star{A}\end{pmatrix} : 
	a \in \KK, A \in \KK' \right\} 
\end{equation}
When not stated otherwise, we assume $\KK'=\OO'$ and $\KK=\OO$, as all other
cases are special cases of this one.
The nice thing about this representation is that the negative of the
determinant on $\Vtwo$ coincides with the norm on $\KK'\oplus\KK$:
\begin{equation}
|\XX|^2 = -\det(\XX) = -(-A\star{A} - \bar{a}a) = |A|^2 + |a|^2 
\label{det}
\end{equation}
Clearly, this norm has signature $(\kap)$, so both $\SO(\kap)$ and its double
cover, the spin group $\Spin(\kap)$ will act on $\Vtwo$.  In the next section,
we will see how to write this representation using matrices over the
composition algebras, thanks to our Clifford representation.

There is a similar construction using the vector space
\begin{equation}
J = \left\{ \begin{pmatrix}A&\bar{a}\\a&\star{A}\end{pmatrix} :
	a \in \KK, A \in \KK' \right\} 
\label{Jordan}
\end{equation}
which provides another representation of $\KK'\oplus\KK$ (as a vector
space).  Remarkably, matrices of the form~(\ref{Jordan}), unlike those of the
form~(\ref{defx}), close under multiplication; not only do such matrices
satisfy their characteristic equation, the resulting algebra is a Jordan
algebra.

Consider now $4\times4$ matrices of the form
\begin{equation}
\PP
  = \Gamma(\XX)
  = \begin{pmatrix}0&\XX\\\widetilde{\XX}&0\end{pmatrix}
\label{Pdef}
\end{equation}
where tilde represents trace reversal,
\begin{equation}
\widetilde{\XX} = \XX - \text{tr}(\XX)\,\II,
\end{equation}
and where the map $\Gamma$ is implicitly defined by~(\ref{Pdef}).  A
straightforward computation using the commutativity of $\KK$ with $\KK'$ shows
that
\begin{equation}
\{\PP,\QQ\} = \PP\QQ + \QQ\PP = 2g(\PP,\QQ)\,\II
\label{CliffordID}
\end{equation}
where $g$ is the inner product obtained by polarizing
$g(\Gamma(\XX),\Gamma(\XX))=-\det(\XX)$ and $\II$ is the identity matrix.
These are precisely the anticommutation relations necessary to give a
representation of the real Clifford algebra $\Cl(12,4)$ (in the case of
$\OO'\otimes\OO$), and $\Cl(\kap)$ in general.

We would therefore like to identify $\PP$ as an element of the Clifford
algebra $\Cl(\kap)$.  However, Clifford algebras are associative, so our
algebra must also be associative.  Since the octonions are not associative,
neither are matrix algebras over the octonions, at least not as matrix
algebras.  The resolution to this puzzle is to always consider octonionic
``matrix algebras'' as linear transformations acting on some vector space, and
to use composition, rather than matrix multiplication, as the product operation.
This construction always yields an associative algebra, since composition
proceeds in a fixed order, from the inside out.

Let's start again.  Recall that $\Ksuper=\KKprime$, and consider the
space $\End(\Kfour)$ of linear maps on $\Kfour$, the set of \hbox{$4\times4$}
matrices with elements in $\Ksuper$.  The matrix $\PP$ can be identified with
the element $\PP_L\in\End(\Kfour)$, where
\begin{equation}
\PP_L(\QQ) = \PP\QQ
\label{PLdef}
\end{equation}
for $\QQ\in\Kfour$.  We have therefore constructed a map $\Gamma_L$
from $\Vtwo$ to $\End(\Kfour)$, given by
\begin{equation}
\Gamma_L(\XX) = \PP_L
\end{equation}
where $\XX$, $\PP$, and $\PP_L$ are defined by~(\ref{defx}), (\ref{Pdef}),
and~(\ref{PLdef}), respectively.  Multiplication in $\End(\Kfour)$ is given by
composition and is associative; under this operation, we claim that the vector
space $\Vfour=\Gamma_L(\Vtwo)$ generates the Clifford algebra
$\Cl(\Vfour)=\Cl(\kap)$, as we now show.

%

\begin{lemma}
If $\PP_L\in\Gamma_L(\Vtwo)$, then
\begin{equation}
(\PP_L)^2 = (\PP^2)_L
\end{equation}
that is, for any $\QQ\in\Kfour$,
\begin{equation}
\PP(\PP\QQ) = \PP^2\QQ
\label{Palt}
\end{equation}
\label{Paltlemma}
\end{lemma}

\begin{proof}\vspace{-0.25in}
Direct computation, using using the alternativity of both $\KK'$ and $\KK$.
\end{proof}

\goodbreak
\begin{theorem}
The subalgebra of $\End(\Kfour)$ generated by $\Gamma_L(\Vtwo)$ is a Clifford
algebra, that is, $\Cl(\Vfour)=\Cl(\kap)$.
\label{thm1}
\end{theorem}

\begin{proof}
Since
\begin{equation}
\Gamma(\XX)^2 = -\det(\XX)\,\II
\label{GammaSq}
\end{equation}
we also have
\begin{equation}
\Gamma_L(\XX)^2 = |\XX|
\label{CliffordID4}
\end{equation}
where $|\XX|=-\det(\XX)$, and where there is an implicit identity operator on
the right-hand side of~(\ref{CliffordID4}).  We can now polarize either of
these expressions to yield
\begin{equation}
\PP(\QQ\Rmat) + \QQ(\PP\Rmat)
 = (\PP\QQ+\QQ\PP)\,\Rmat
 = 2g(\PP,\QQ)\,\Rmat
\label{CliffordID2}
\end{equation}
with $\PP,\QQ,\Rmat\in\Kfour$.  That is, we have
\begin{equation}
\{\PP_L,\QQ_L\} = \{\PP,\QQ\}_L
\label{CliffordID3}
\end{equation}
which, together with the Clifford identity~(\ref{CliffordID}) and the
associativity of $\End(\Kfour)$, can now be used to establish that the
algebra generated by $\Gamma_L(\Vtwo)$ is the Clifford algebra $\Cl(\Vfour)$.
\end{proof}

\section{Spin groups from composition algebras}
\label{ortho}

Representations of Clifford algebras yield representations of the
corresponding orthogonal groups, or at least their double cover, using a
well-known construction.  Applying this to our representation of $\Cl(\kap)$
gives us a representation of $\Spin(\kap)$ using matrices over composition
algebras.  Our use of nonassociative algebras in our representation
requires care, yet we shall see that we are in the best possible situation:
the action of generators of $\Spin(\kap)$ can be expressed entirely in terms
of matrix multiplication over composition algebras, associated in a fixed
order.

First, let us give a brief overview of the general construction.  Given a
vector space $V$ equipped with quadratic form, the unit vectors generate a
subgroup of $\Cl(V)$ called the \define{pin group} $\Pin(V)$.  This group is
the double cover of $\O(V)$, which means there is a 2-to-1 and onto
homomorphism
\begin{equation}
R \maps \Pin(V) \to \O(V) .
\label{R1}
\end{equation}
The \define{spin group} $\Spin(V)$ is the subgroup of $\Pin(V)$ generated by
products of pairs of unit vectors.  It is the double cover of $\SO(V)$, which
means there is a 2-to-1 and onto homomorphism
\begin{equation}
R \maps \Spin(V) \to \SO(V) .
\label{R2}
\end{equation}
The map~(\ref{R2}) is just the restriction of~(\ref{R1}) to $\Spin(V)$, so we
give it the same name.

We will describe $R$ by saying what it does to generators.  Let $w$ be a unit
vector, that is, a vector $w\in V$ such that $|w|^2=\pm1$, where $|w|^2$
denotes the action of the quadratic form on $w$.  Define $R_w \maps V \to V$
to be the \define{reflection along $w$}: the linear map taking $w$ to $-w$ and
fixing the hyperplane orthogonal to $w$.  At the heart of the connection
between Clifford algebras and geometry, we have the fact that $R_w$ can be
written solely with operations in the Clifford algebra:
\begin{equation}
R_w (v) = -wvw^{-1}
\label{reflect}
\end{equation}
Checking this using the Clifford relation is a straightforward calculation,
which we nonetheless do here because it plays a role in what follows.
Clearly, $R_w(w)=-w$.  If $v$ is orthogonal to $w$, the Clifford relation tells
us $wv+vw=0$, so $w$ and $v$ anticommute, and $R_w(v)=v$.  Hence, $R_w$ is the
unique linear map taking $w$ to $-w$ and fixing the hyperplane orthogonal to
$w$.

In fact, $R$ extends from a map on the generators of $\Pin(V)$, taking $w$ to
$R_w$, to a homomorphism.  This homomorphism is 2-to-1, as suggested by the
fact that $R_w=R_{-w}$.  Since it is well known that $\O(V)$ is generated by
reflections of the form $R_w$, and $\SO(V)$ by products of pairs of these,
this homomorphism is clearly onto.

In~(\ref{reflect}), we expressed reflection using Clifford multiplication of
vectors.  Yet the endomorphisms in $\Vfour=\Gamma_L(\Vtwo)$ correspond to
$4\times4$ matrices in $\Gamma(\Vtwo)$, so we can also multiply them as
matrices, although this product is no longer associative.  Remarkably, thanks
to the matrix form of the Clifford relation, this gives us another way to
express reflections.

We begin by noting that the elements of $\XX$, and hence of $\PP$, commute,
since they jointly contain only one independent imaginary direction in each of
$\KK$ and $\KK'$.  Thus, there is no difficulty defining the determinants of
these matrices as usual.  Since $\PP^2$ is proportional to the identity matrix
by~(\ref{GammaSq}), computing
\begin{equation}
\det\bigl(\Gamma(\XX)\bigr) = (\det\XX)^2
\end{equation}
shows that $\PP^{-1}$ is proportional to $\PP$ so long as $\det\PP\ne0$.

\begin{lemma}
Let $\PP,\QQ\in\Gamma(\Vtwo)$, with $\det\PP\ne0$.  Then
\begin{equation}
(\PP \QQ) \PP^{-1} = \PP ( \QQ \PP^{-1} )
\label{PQP}
\end{equation}
and this matrix, which we denote $\PP\QQ\PP^{-1}$, also lies in $\Gamma(\Vtwo)$.
\end{lemma}

\begin{proof}
By the discussion above, $\PP^{-1}$ is proportional to $\PP$, so that the
elements of $\PP$, $\QQ$, and~$\PP^{-1}$ jointly contain only two independent
imaginary directions in each of $\KK$ and $\KK'$.  Thus, there are no
associativity issues when multiplying these matrices, which
establishes~(\ref{PQP}).  Direct computation further establishes the fact that
$\PP\QQ\PP^{-1}\in\Gamma(\Vtwo)$.
\end{proof}

\begin{lemma}
Let $\PP,\QQ\in\Gamma(\Vtwo)$ with $|\PP|=1$, so that $\PP_L,\QQ_L\in\Vfour$.
Then
\begin{equation}
R_{\PP_L}(\QQ_L) = - \left( \PP\QQ\PP^{-1} \right)_L .
\end{equation}
\label{Rlemma}
\end{lemma}

\begin{proof}
Given that $\PP^2$ is a multiple of the identity, it is enough to show that
\begin{equation}
\PP_L\circ\QQ_L\circ\PP_L = (\PP\QQ\PP)_L
\end{equation}
in $\Kfour$, that is, that
\begin{equation}
\PP(\QQ(\PP(\Rmat))) = (\PP\QQ\PP)(\Rmat)
\label{Moufang}
\end{equation}
for $\Rmat\in\Gamma(\Vtwo)$.  But~(\ref{Moufang}) follows immediately from the
Moufang identity
\begin{equation}
p(q(p(r))) = (pqp)r
\end{equation}
and the antisymmetry of the associator, which hold in both $\KK$ and $\KK'$.
\end{proof}

Lemma~\ref{Rlemma} is the key computation, as it immediately gives us a
representation of $\Pin(\Vfour)$ using matrices over division algebras, and
allows us to finally identify $\Vfour$ with $\Gamma(\Vtwo)$.  We continue to
write $\PP\QQ$ for the matrix product in $\Gamma(\Vtwo)$, and introduce the
notation $\PP\cdot\QQ$ for the Clifford product in $\Vfour$, that is, as
shorthand for $\PP_L\circ\QQ_L$.

\begin{lemma}
There is a homomorphism
\begin{equation}
R \maps \Pin(\Vfour) \to \O(\Vfour)
\label{Omap}
\end{equation}
which sends unit vectors $\PP \in \Vfour$ to the element $R_\PP$ of
$\O(\Vfour)$ given by:
\begin{equation}
R_\PP(\QQ) = -\PP\QQ\PP^{-1}, \quad \QQ \in \Vfour
\end{equation}
and sends a general element $g = \PP_1 \cdot \PP_2 \cdot \cdots \cdot \PP_n$
in $\Pin(\Vfour)$ to the element of $\O(\Vfour)$ given by:
\begin{equation}
R_g(\QQ) = (-1)^n \PP_1(\PP_2( \cdots (\PP_n \QQ \PP^{-1}_n )
		\cdots) \PP_2^{-1} ) \PP_1^{-1} .
\end{equation}
\label{L2}
\end{lemma}

\begin{proof}
This result follows immediately from the definition of the homomorphism $R$
and the fact that we can use Lemma~\ref{Rlemma} to express $R$ using matrix
multiplication.
\end{proof}

Restricting Lemma~\ref{L2} to the spin group, we get the usual representation
of $\Spin(\Vfour)$ on $\Vfour$, expressed using matrices over division
algebras.

\begin{lemma}
There is a homomorphism
\begin{equation}
R \maps \Spin(\Vfour) \to \SO(\Vfour)
\label{Smap}
\end{equation}
which sends a product of unit vectors $g = \PP_1\cdot\PP_2$ in $\Spin(\Vfour)$
to the element $R_g$ of $\SO(\Vfour)$ given by:
\begin{equation}
R_g(\QQ) = \PP_1(\PP_2\QQ\PP_2^{-1})\PP_1^{-1} .
\label{double}
\end{equation}
\label{L3}
\end{lemma}

\begin{proof}
The homomorphism~(\ref{Smap}) is just the restriction of~(\ref{Omap}) to the
spin group.
\end{proof}

We have proved:

\begin{theorem}
The second-order homogeneous elements of $\Cl(\kap)$ generate an action of
$\SO(\kap)$ on $\Vfour=\Gamma(\Vtwo)$.
\label{T1}
\end{theorem}

\section{\boldmath An Explicit Construction of $\SO(\kap)$}
\label{explicit}

We now implement the construction in the previous section, obtaining an
explicit construction of the generators of $\SO(\kap)$ in a preferred basis.

We can expand elements $\XX\in\Vtwo$ in terms of a basis as
\begin{equation}
\XX = x^\inda\SS_\inda
\label{Xdef}
\end{equation}
where we have set
\begin{equation}
x^\inda = 
\begin{cases}
a_\inda & (1\le \inda\le8)\\
A_{\inda-8} & (9\le \inda\le16)
\end{cases}
\end{equation}
and where there is an implicit sum over the index $\inda$, which takes on
values between~$1$ and~$16$ as appropriate for the case being considered.
Equation~(\ref{Xdef}) defines the \textit{generalized Pauli matrices}
$\SS_\inda$, which are given this name because $\SS_1$, $\SS_2$, and $\SS_9$
are just the usual Pauli spin matrices.  We can further write
\begin{equation}
\PP
  = \Gamma(\XX)
  = x^\inda\GG_\inda ,
\end{equation}
which implicitly defines the gamma matrices $\GG_\inda=\Gamma(\SS_p)$.  (The
only $\SS_\inda$ which are affected by trace reversal are those containing an
imaginary element of~$\KK'$, which are imaginary multiples of the identity
matrix.)
Direct computation shows that
\begin{equation}
\{\GG_\inda,\GG_\indb\} = 2g_{\inda\indb}\II
\label{CliffordID5}
\end{equation}
where
\begin{equation}
g_{\inda\indb}
  = \begin{cases}
	0&\inda\neq\indb\\
	1&1\le\inda=\indb\le12\}\\
	-1&13\le\inda=\indb\le16\}
    \end{cases}
\end{equation}
and we have recovered~(\ref{CliffordID}).

The elements of $\Vfour$ are the homogeneous linear elements of $\Cl(\kap)$.
In the associative case, the homogeneous quadratic elements of $\Cl(\kap)$
would act on $\Vfour$ as generators of $\SO(\kap)$ via the map
\begin{equation}
\label{paction}
\PP \longmapsto \MM_{\inda\indb}\PP\MM_{\inda\indb}^{-1}
\end{equation}
where
\begin{equation}
\MM_{\inda\indb}
  = \exp\left(-\GG_\inda\GG_\indb\>\frac{\theta}{2}\right)
\label{mdef}
\end{equation}
and with $\PP = x^\inda\GG_\inda$ as above.  We introduce the notation $e_a$
for the octonionic and split octonionic units, defined so that
\begin{equation}
a+A = x^a e_a
\end{equation}
in analogy with~(\ref{Xdef}), and we consider first the case where
$[e_\inda,e_\indb,e_\indc]=0$, with $\inda,\indb,\indc$ assumed to be
distinct.  Then the Clifford identity~(\ref{CliffordID5}) implies that
\begin{align}
\GG_\inda\GG_\inda &= \pm \II,
  \label{prop1}\\
(\GG_\inda\GG_\indb)\GG_\indc
  &= \GG_\indc(\GG_\inda\GG_\indb),
  \label{prop2}\\
(\GG_\inda\GG_\indb)\GG_\indb
  &= (\GG_\indb)^2\GG_\inda
   = g_{\indb\indb}\GG_\inda,
  \label{prop3}\\
(\GG_\inda\GG_\indb)\GG_\inda
  &= -(\GG_\inda)^2\GG_\indb
   = -g_{\inda\inda}\GG_\indb,
  \label{prop4}\\
(\GG_\inda\GG_\indb)^2
  &= -\GG_\inda^2\GG_\indb^2
   = \pm \II
  \label{prop5},
\end{align}
With these observations, we compute
\begin{equation}
\label{action4}
\MM_{\inda\indb}\PP\MM_{\inda\indb}^{-1}
  = \exp\left(-\GG_\inda\GG_\indb\>\frac{\theta}{2}\right)
	\left(x^\indc\GG_\indc\right)
	\exp\left(\GG_\inda\GG_\indb\>\frac{\theta}{2}\right).
\end{equation}
From~(\ref{prop1}), if
$\inda=\indb$, then $\MM_{\inda\indb}$ is a real multiple of the identity
matrix, which therefore leaves $\PP$ unchanged under the
action~(\ref{paction}).  On the other hand, if $\inda\neq\indb$,
properties~(\ref{prop2})--(\ref{prop4}) imply that $\MM_{\inda\indb}$
commutes with all but two of the matrices $\GG_\indc$.  We therefore
have
\begin{equation}
\GG_\indc\MM_{\inda\indb}^{-1}
  = \begin{cases}
	\MM_{\inda\indb}\GG_\indc,
		&\indc=\inda\text{ or }\indc=\indb\\
	\MM^{-1}_{\inda\indb}\GG_\indc,
		&\inda\neq\indc\neq\indb
    \end{cases}
\label{maction}
\end{equation}
so that the action of $\MM_{\inda\indb}$ on $\PP$ affects only the
$\inda\indb$ subspace.  To see what that action is, we first note that if
$\AAA^2 = \pm\II$ then
\begin{equation}
\label{euler}
\exp\left(\AAA\alpha\right)
  = \II\,c(\alpha) + \AAA\,s(\alpha)
  = \begin{cases}
	\II\,\cosh(\alpha) + \AAA\,\sinh(\alpha),&\AAA^2 = \II\\
	\II\,\cos(\alpha) + \AAA\,\sin(\alpha),&\AAA^2 = -\II
    \end{cases}
\end{equation}
where the second equality can be regarded as defining the functions~$c$
and~$s$.  Inserting~(\ref{maction}) and~(\ref{euler}) into~(\ref{action4}), we
obtain
\begin{align}
\label{action}
\MM_{\inda\indb} \left(
	x^\inda\GG_\inda + x^\indb\GG_\indb
	\right) \MM_{\inda\indb}^{-1}
  &= \left(\MM_{\inda\indb}\right)^2
	\left(x^\inda\GG_\inda + x^\indb\GG_\indb\right)
  \nonumber\\
  &= \exp\left(-\GG_\inda\GG_\indb\theta\right)
	\left(x^\inda\GG_\inda + x^\indb\GG_\indb\right)
  \nonumber\\
  &= \bigl(
	\II\,c(\theta) - \GG_\inda\GG_\indb\,s(\theta)
     \bigr)
	\left(x^\inda\GG_\inda + x^\indb\GG_\indb\right)
  \nonumber\\
  &= \left(x^\inda c(\theta) - x^\indb s(\theta)g_{\indb\indb}\right)
	\GG_\inda
	+ \left(x^\indb c(\theta) + x^\inda s(\theta)g_{\inda\inda}\right)
	  \GG_\indb.
\end{align}
Thus, the action~(\ref{paction}) is either a rotation or a boost in the
$\inda\indb$-plane, depending on whether
\begin{equation}
(\GG_\inda\GG_\indb)^2 = \pm\II
\end{equation}
More precisely, if $\inda$ is spacelike ($g_{\inda\inda}=1$),
then~(\ref{paction}) corresponds to a rotation by~$\theta$ from~$\inda$
to~$\indb$ if $\indb$ is also spacelike, or to a boost in the $\inda$
direction if $\indb$ is timelike ($g_{\indb\indb}=-1$), whereas if $\inda$ is
timelike, the rotation (if $\indb$ is also timelike) goes from $\indb$
to~$\inda$, and the boost (if $\indb$ is spacelike) is in the negative $\inda$
direction.

If $\KKprime=\HH'\otimes\HH$ (or any of its subalgebras), we're done:
since transformations of the form~(\ref{paction}) preserve the determinant of
$\PP$, it is clear from~(\ref{det}) that we have constructed $\SO(6,2)$
(or one of its subgroups).

What about the nonassociative case?
We can no longer use~(\ref{prop2}), which now contains an extra minus sign.  A
different strategy is needed.

If $e_\inda$, $e_\indb$ commute, then they also associate with every basis
unit, that is
\begin{equation}
[e_\inda, e_\indb] = 0 \Longrightarrow [e_\inda, e_\indb,e_\indc]=0
\label{commass}
\end{equation}
and the argument above leads to~(\ref{action}) as before.  We therefore assume
that $e_\inda$, $e_\indb$ anticommute, the only other possibility; in this
case, $e_\inda$, $e_\indb$ are imaginary basis units that either both lie in
in $\OO$, or in $\OO'$.
As before, we seek a transformation that acts only on the $\inda\indb$
subspace.  But in this case, we have:

\begin{lemma}
$\Gamma_\inda\Gamma_\indb\Gamma_\inda^{-1} = e_\inda\Gamma_\indb e_\inda^{-1}$
for $\inda\ne\indb\in\{2,...,8\}$ or $\inda\ne\indb\in\{10,...,16\}$.
\label{LPe}
\end{lemma}

\begin{proof}
Use the Clifford identity and the fact that $\Gamma_\inda^2=1$.
\end{proof}

Consider therefore the transformation
\begin{equation}
\label{flip}
\PP \longmapsto e_\inda\PP e_\inda^{-1}
\end{equation}
which preserves directions corresponding to units $e_\indb$ that commute with
$e_\inda$, and reverses the rest, which anticommute with $e_\inda$.  We call
this transformation a \textit{flip} about $e_\inda$; any imaginary unit can be
used, not just basis units.  If we compose flips about any two units in the
$\inda\indb$ plane, then all directions orthogonal to this plane are either
completely unchanged, or flipped twice, and hence invariant under the combined
transformation.  Such double flips therefore affect only the $\inda\indb$
plane.
\footnote{We use flips rather than reflections because flips are themselves
rotations, whereas reflections are not.}

The rest is easy.  We \textit{nest} two flips, replacing~(\ref{paction}) by
\begin{equation}
\PP \longmapsto
	\MM_2\left(\MM_1\PP\MM_1^{-1}\right)\MM_2^{-1}
\label{nest}
\end{equation}
where
\begin{align}
\MM_1 &= -e_\inda\,\II \nonumber\\
\MM_2
  &= \left(e_\inda\,c(\halfang)+e_\indb\,s(\halfang)\right)\,\II
  \nonumber\\
  &= \begin{cases}
	\left(e_\inda\cosh(\halfang) + e_\indb\,\sinh(\halfang)\right)\II,
		&(e_\inda e_\indb)^2 = 1 \\
	\left(e_\inda\cos(\halfang)+e_\indb\,\sin(\halfang)\right)\II,
		&(e_\inda e_\indb)^2 = -1
     \end{cases}
\label{M12}
\end{align}
Using the relationships
\begin{align}
\bigl(e_\inda c(\alpha)+e_\indb s(\alpha)\bigr)^2 
   = e_\inda^2 c^2(\alpha)+e_\indb^2 s^2(\alpha)
  &= e_\inda^2 = -g_{\inda\inda}\\
e_\inda^2 c^2(\alpha)-e_\indb^2 s^2(\alpha)
  &= -g_{\inda\inda} c(2\alpha) \\
2s(\alpha)c(\alpha) &= s(2\alpha)
\end{align}
we now compute
\begin{align}
&\MM_2\left( \MM_1 \left(
	x^\inda\GG_\inda + x^\indb\GG_\indb
	\right) \MM_1^{-1} \right) \MM_2^{-1}
   = \MM_2\left( 
	x^\inda\GG_\inda - x^\indb\GG_\indb
	\right) \MM_2^{-1} \nonumber\\
  &\qquad\qquad
   = \left(e_\inda\,c(\halfang)+e_\indb\,s(\halfang)\right) \left( 
	x^\inda\GG_\inda - x^\indb\GG_\indb
	\right) \left(e_\inda\,c(\halfang)+e_\indb\,s(\halfang)\right)
	(-g_{\inda\inda})
	\nonumber\\
  &\qquad\qquad
   = \left(x^\inda c(\theta) - x^\indb s(\theta)\,
	g_{\inda\inda}\,g_{\indb\indb}\right) \GG_\inda
	+ \left(x^\indb c(\theta) + x^\inda s(\theta)\right) \GG_\indb.
\label{nonass}
\end{align}
and we have constructed the desired rotation in the $\inda\indb$ plane.

We also have
\begin{equation}
\GG_\inda\GG_\indb = -e_\inda e_\indb\,\II
\qquad\qquad ([e_\inda,e_\indb]\ne0)
\end{equation}
so in the associative case (with $e_\inda$, $e_\indb$ anticommuting), we have
\begin{equation}
\MM_2\MM_1
  = \left( g_{\inda\inda} c(\halfang) + e_\inda e_\indb s(\halfang) \right)
	\II
  = g_{\inda\inda} \exp\left( -g_{\inda\inda}
	\GG_\inda\GG_\indb\>\frac{\theta}{2}\right)
\end{equation}
which differs from $M_{\inda\indb}$ only in replacing $\theta$ by $-\theta$
(and an irrelevant overall sign) if $g_{\inda\inda}=-1$.  In other words, the
nested action~(\ref{nest}) does indeed reduce to the standard
action~(\ref{paction}) in the associative case, up to the orientations of the
transformations.  In this sense,~(\ref{nest}) is the nonassociative
generalization of the process of exponentiating homogeneous elements of the
Clifford algebra in order to obtain rotations in the orthogonal group.

We therefore use~(\ref{paction}) if $e_\inda$ and $e_\indb$ commute,
and~(\ref{nest}) if they don't.  Since both of these transformations preserve
the determinant of $\PP$, it is clear from~(\ref{det}) that we have
constructed $\SO(\kap)$ from $\Cl(\kap)$.

\begin{theorem}
The nested flips~(\ref{nest})--(\ref{M12}) generate an action of
\hbox{$SO(\kap)$} on $\Vfour$.
\label{T3}
\end{theorem}

\begin{proof}
If $[e_\inda,e_\indb]=0$, Lemma~\ref{LPe} and~(\ref{commass}) imply
that~(\ref{nest}) reduces to~(\ref{paction}) (up to an irrelevant sign), and
this action was shown in~(\ref{action}) to be a rotation in the $\inda\indb$
plane.
If $[e_\inda,e_\indb]\ne0$, then~(\ref{nonass}) shows that this action is
again a rotation in the $\inda\indb$ plane.
Since we have constructed rotations in all coordinate planes, we can combine
them using generalized Euler angles to produce any desired group element.
\end{proof}

Equivalently, Lemma~\ref{LPe} and the reduction of~(\ref{nest})
to~(\ref{paction}) in the associative case, together with the equivalence of
nested reflections and nested flips, show that Theorem~\ref{T3} follows from
Theorem~\ref{T1}.  That is, the action~(\ref{nest}) of nested flips of the
form~(\ref{M12}) agrees with the action of the double
reflection~(\ref{double}), with $\PP_1=\Gamma(\SS_\inda)$ and
$\PP_2=\Gamma(\SS_\indb)$.

\section{\boldmath The Group $\SU(2,\KKprime)$}
\label{su2}

So far we have considered transformations of the form~(\ref{paction})
and~(\ref{nest}) acting on $\PP$.  In light of the off-diagonal structure
of the matrices $\{\GG_\inda\}$, we can also consider the effect
these transformations have on $\XX$.  First, we observe that
trace-reversal of $\XX$ corresponds to conjugation in $\KK'$, that is,
\begin{equation}
\widetilde{\SS_\inda} = \star{\SS}_\inda.
\end{equation}
The matrices $\GG_\inda\GG_\indb$ then take the form
\begin{equation}
\GG_\inda\GG_\indb
  = \begin{pmatrix}
	\SS_\inda\star{\SS}_\indb&0\\
	0&\star{\SS}_\inda\SS_\indb
    \end{pmatrix}
\end{equation}
and, in particular,
\begin{equation}
\exp\left(\GG_\inda\GG_\indb\>\frac{\theta}{2}\right)
  = \begin{pmatrix}
	\exp\left( \SS_\inda\star{\SS}_\indb\>\frac{\theta}{2} \right)&0\\
\noalign{\smallskip}
	0&\exp\left( \star{\SS}_\inda\SS_\indb\>\frac{\theta}{2} \right)
    \end{pmatrix},
\end{equation}
so we can write
\begin{align}
&\exp\left(-\GG_\inda\GG_\indb\>\frac{\theta}{2}\right)
	\>\PP\>
	\exp\left(\GG_\inda\GG_\indb\>\frac{\theta}{2}\right)
	\nonumber\\
&\qquad\qquad=
    \begin{pmatrix}
	0&\exp\left( -\SS_\inda\star{\SS}_\indb\>\frac{\theta}{2} \right) \XX
	  \exp\left( \star{\SS}_\inda\SS_\indb\>\frac{\theta}{2} \right)\\
\noalign{\smallskip}
	\exp\left( -\star{\SS}_\inda\SS_\indb\>\frac{\theta}{2} \right)
		\widetilde{\XX}
	\exp\left(\SS_\inda\star{\SS}_\indb\>\frac{\theta}{2} \right)&0
    \end{pmatrix}.
\end{align}
The $4\times4$ action~(\ref{paction}) acting on $\PP$ is thus equivalent to the
action
\begin{equation}
\XX \longmapsto
	\exp\left( -\SS_\inda\star{\SS}_\indb\>\frac{\theta}{2} \right) \XX
	\exp\left( \star{\SS}_\inda\SS_\indb\>\frac{\theta}{2} \right). 
\label{xaction}
\end{equation}
on $\XX$.

\goodbreak

Transformations of the form~(\ref{nest}) are even easier, since each of
$\MM_1$ and $\MM_2$ are multiples of the identity matrix $\II$.
These transformations therefore act on $\XX$ via
\begin{equation}
\XX \longmapsto \MM_2\left(\MM_1\XX\MM_1^{-1}\right)\MM_2^{-1}
\end{equation}
where $\MM_1$, $\MM_2$ are given by~(\ref{M12}), but with $\II$
now denoting the $2\times2$ identity matrix.

Since $\XX$ is Hermitian with respect to $\KK$, and since that condition is
preserved by~(\ref{xaction}), we have realized $\SO(\kap)$ in terms of
(possibly nested) determinant-preserving transformations involving $2\times2$
matrices over $\KKprime$.  This $2\times2$ representation of $\SO(\kap)$
therefore deserves the name $\SU(2,\KKprime)$.

Further justification for the name $\SU(2,\KKprime)$ comes from the
realization that nested transformations of the form~(\ref{nest}) yield
rotations wholly within $\Im\OO$ or $\Im\OO'$.  All other rotations can be
handled without any associativity issues, yielding for instance the standard
matrix description of $\SU(2,\HH'\otimes\HH)$.  But any rotation wholly within
$\Im\OO$ or $\Im\OO'$ can be obtained as a composition of rotations in other
coordinate planes.  In this sense, what we have called $\SU(2,\KKprime)$ is
the \textit{closure} of the set of matrix transformations that preserve the
determinant of $\XX$.  Equivalently, at the Lie algebra level,
$\sa_2(\KKprime)$ is not a Lie algebra, since it is not closed.  However, its
closure is precisely the infinitesimal version of our $\SU(2,\KKprime)$.

\section{Discussion}
\label{discuss}

We have given two division algebra representations of the groups $\SO(\kap)$
that appear in the $2\times2$ magic square in Table~\ref{2x2gp}, namely the
$4\times4$ representation constructed from quadratic elements of the Clifford
algebra $\Cl(\kap)$ in Section~\ref{ortho}, and the $2\times2$ representation
$\SU(2,\KKprime)$ in Section~\ref{su2}.  Each of these representations
provides a unified description of the $2\times2$ magic square of Lie groups,
in the spirit of Vinberg's description~(\ref{Vin3}) of the Freudenthal--Tits
magic square of Lie algebras.

Our work is in the same spirit as that of Manogue and Schray~\cite{Lorentz},
leading to~(\ref{so91}), but there are some subtle differences.  In effect,
all we have done in this case (the second row of Table~\ref{2x2gp}) is to
multiply the time direction by the split complex unit $L$.  This changes very
little in terms of formal computation, but allows room for generalization by
adding additional split units, thus enlarging $\CC'$ to $\HH'$ or $\OO'$.  It
also has the advantage of turning our representation space $\{\XX\}$ into a
collection of matrices whose real trace vanishes, as is to be expected for a
representation of a unitary group.

However, unlike the transformations constructed by Manogue and Schray, our
transformations~(\ref{xaction}) do not appear to have the general form
\begin{equation}
\XX\longmapsto\MM\XX\MM^\dagger,
\end{equation}
even if we restrict the dagger operation to include conjugation in just one of
$\KK'$ or $\KK$.  This point remains under investigation, but seems a small
price to pay for a unified description of the full magic square.

Our use of \textit{nested flips} in Section~\ref{explicit} is again motivated by
the work of Manogue and Schray~\cite{Lorentz}, but yet again there are some
subtle differences.  Over $\OO$, as in~\cite{Lorentz}, the transformations
affecting the imaginary units are all rotations in $\SO(7)$; over $\OO'$, by
contrast, these transformations lie in $\SO(3,4)$, and some are boosts.  It is
straightforward to connect a flip affecting an even number of spatial
directions to the identity: Simply rotate these directions pairwise by $\pi$.
Not so for our transformations~(\ref{nest}) in the case where
$e_\inda,e_\indb\in\OO'$, since we must count separately the number of
spacelike and timelike directions affected, which could both be odd.  It would
be straightforward to expand our flips so that they act nontrivially on an
even number of spacelike directions (and therefore also on an even number of
timelike directions), and such flips would then be connected to the identity
using pairwise rotation.  However, these component flips would no longer take
the simple form~(\ref{M12}).

In future work, we hope to extend this approach to the $3\times3$ magic square
in Table~\ref{3x3gp}, and conjecture that the end result will be a unified
description of the form $\SU(3,\KKprime)$.  It appears to be
straightforward to reinterpret our previous description~(\ref{e6}) of
$E_{6(-26)}$~\cite{Denver,York,Structure,Sub} so as to also imply that
\begin{equation}
E_{6(-26)}\equiv\SU(3,\CC'\otimes\OO)
\end{equation}
but the conjectured interpretations
\begin{align}
E_{7(-25)}&\equiv\SU(3,\HH'\otimes\OO) \\
E_{8(-24)}&\equiv\SU(3,\OO'\otimes\OO)
\end{align}
would be new.

\section*{Acknowledgments}

We thank Corinne Manogue, Tony Sudbery, and Rob Wilson for helpful comments.
The completion of this paper was made possible in part through the support of
a grant from the John Templeton Foundation, and the hospitality of the
University of Denver during the 3rd Mile High Conference on Nonassociative
Mathematics.


\end{document}